\documentclass[a4paper,11pt,oneside,openany,reqno]{amsart}

\usepackage{latexsym,amsmath,amssymb,amsfonts}
\usepackage{bm}
\usepackage[dvipdfmx]{graphicx}
\usepackage{subfigure}
\usepackage{verbatim}
\usepackage{wrapfig}
\usepackage{ascmac}
\usepackage{cases}
\usepackage{amsthm}
\usepackage{color}

\usepackage{comment}
\usepackage{ulem}
\usepackage {cancel}

  \makeatletter
    
    \@addtoreset{equation}{section}
  \makeatother
\newtheorem{theo}{Theorem}[section]

\newtheorem{lemm}[theo]{Lemma}

\newcounter{c}
\newcounter{d}
\newcounter{b}
\newcommand{\cc}[1][]{\refstepcounter{c}#1\arabic{c}} 
\newcommand{\ce}[1][]{\refstepcounter{d}#1\arabic{d}}
\newcommand{\be}[1][]{\refstepcounter{b}#1\arabic{b}}

\newcommand{\Na}{\mathbb{N}} 
\newcommand{\Z}{\mathbb{Z}} 
\newcommand{\R}{\mathbb{R}} 
\newcommand{\T}{\mathbb {T}} 

\newcommand{\e}{\varepsilon} 
\newcommand{\di}{\mathrm{div}} 

\title[Prescribed mean curvature]{Prescribed mean curvature equation on torus}
\author{Yuki Tsukamoto}
\address{Department of Mathematics, Tokyo Institute of Technology,
	152-8551, Tokyo, Japan}
\email{tsukamoto.y.ag@m.titech.ac.jp}
\date{}

\begin{document}
\maketitle
\begin{abstract}	
Prescribed mean curvature problems on the torus has been considered in one dimension.
In this paper, we prove the existence of a graph on the $n$-dimensional torus $\T^n$, 
the mean curvature vector of which equals the normal component
of a given vector field satisfying suitable conditions for a Sobolev norm,
the integrated value, and monotonicity.
\end{abstract}

	\section{Introduction}

In this paper, we consider the following prescribed mean 
curvature problem on torus $\T^n:=\R^n/ \Z^n$:
\begin{align}
- \di \left(\frac{\nabla u}{\sqrt{1+|\nabla u|^2}} \right)= \nu(\nabla u) \cdot g(x,u(x)) \quad 
\mathrm{on} \  \T^n, \label{first}
\end{align}
where $\nu$ is the unit normal vector of $u$, that is,
$\nu(z)=\frac{1}{\sqrt{1+|z|^2}} (-z, 1)$.
The vector field $g(x,x^{n+1}): \T^n \times \R \to \R^{n+1}$ is given, and we
seek a solution $u$ satisfying \eqref{first}.
The left-hand side of \eqref{first} represents the mean curvature of the graph of $u$,
and the right-hand side is the normal component of the vector field $g$ on the graph.

In the case of Dirichlet conditions of a bounded domain $\Omega \subset \R^n$,
prescribed mean curvature problems have been studied by numerous researchers.
Bergner \cite{B08} solved the Dirichlet problem in the case where the right-hand side of \eqref{first}
is $H=H(x,u,\nu(\nabla u))$ under the assumptions of boundedness ($|H|<\infty$),
monotonicity ($\partial_{n+1} H \geq 0$), and convexity of $\Omega$.
Under the same conditions for the function $H$,
Marquardt \cite{M10} imposed a condition on $\partial\Omega$ depending on $H$
that guarantees the existence of solutions even for a domain $\Omega$ that is not necessarily convex. 
In \cite{T20}, we proved the existence of a solution only under the condition that
the Sobolev norm of $H$ is sufficiently small.
In the case of a compact Riemannian manifold,
Aubin \cite{A98} solved the linear elliptic problem $- \partial_i[a_{ij}(x) \partial_j u] = H(x)$ 
if the integrated value of $H$ is zero. The assumption of the
integrated value plays an important role in the existence of solutions to elliptic equations
on a compact Riemannian manifold.
Denny \cite{D10} solved the quasilinear elliptic problem $- \di (a(u(x)) \nabla u)=H(x)$ on the torus $\T^n$ with $n=2,3$.
Prescribed mean curvature problems on the one-dimensional torus $\left(\frac{u'}{\sqrt{1+(u')^2}}\right)' =H(x,u,u')$ have been investigated 
for a wide variety of conditions $H$ (refer to \cite{F12,L10,L15,P11,P11b,Z15}, for example).

As we noted in \cite{T20}, the motivation for the present study comes from a singular perturbation problem, and we proved the following in \cite{T19}. Suppose a constant $\e>0$ and functions $\phi_\e \in W^{1,2}$ and $g_\e \in W^{1,p}$, with $p>\frac{n+1}{2}$ satisfy 
\begin{align}
&-\varepsilon\Delta \phi_{\varepsilon}+\frac{W'(\phi_{\varepsilon})}{\varepsilon}=\varepsilon\nabla\phi_{\varepsilon}\cdot g_{\varepsilon}, \\
	&\int \left(\frac{\varepsilon|\nabla\phi_{\varepsilon}|^2}{2}
+\frac{W(\phi_{\varepsilon})}{\varepsilon}\right)\,dx+
\|g_{\varepsilon}\|_{W^{1,p}(\tilde\Omega)}\leq C,
\end{align}
where $W$ is a double-well potential such as $W(\phi)=(1-\phi^2)^2$.
Then, the interface $\{\phi_\e = 0 \}$ converges locally in the Hausdorff distance to a
surface having a mean curvature given by $\nu\cdot g $ as $\e \to 0$. Here, $\nu$ is the unit normal vector of the surface, and $g$ is the weak $W^{1,p}$ limit of $g_\e$.
If the surface is represented locally as a graph of a function $u$ on $\T^n$, 
we can observe that $u$ satisfies \eqref{first}.
In this paper, we prove the existence of solutions to \eqref{first} assuming 
that the Sobolev norm of $g$ is sufficiently small, $g^{n+1}$ for the $n+1$th component is monotonous,
and the integrated value of $g^{n+1}$ is zero. 
The following theorem is the main result.

\begin{theo} \label{MT}
	Fix $\frac{n+1}{2}<p<n+1$ and $q =\frac{np}{n+1-p}$. 
	Then, there exists a constant $\e_{\ce \label{ce1}}=\e_{\ref{ce1}}(n,p)>0$
	with the following property.
	If $\e<\e_{\ref{ce1}}$, and $g =(g^1,\ldots, g^n,g^{n+1})=(g',g^{n+1}) \in W^{1,p}(\T^n \times (-1,1);\R^{n+1})$ satisfies \eqref{t1e1}--\eqref{t1e3},
	\begin{align}
	 \|g\|_{W^{1,p}(\T^n \times (-1,1))}&< \e^\frac23, \label{t1e1}\\
      \partial_{n+1} g^{n+1}(x,x^{n+1}) &>  \e+\e^\frac12| \partial_{n+1} g'(x,x^{n+1}) |  ,
    \label{t1e4} \\
     \int_{\T^n} g^{n+1}(x,0) &=0, \label{t1e3} 
	\end{align}
	then there exists a function $u \in W^{2,q}(\T^n)$ such that
	\begin{align}
		- \di \left(\frac{\nabla u}{\sqrt{1+|\nabla u|^2}} \right)= \nu(\nabla u) \cdot g(x,u(x)) \quad 
	\mathrm{on} \  \T^n. \label{t0e1}
		\end{align}
	Moreover, the following inequality holds:
	\begin{align}
	\left\|u - \int_{\T^n} u (y)\ dy\right\|_{W^{2,q}(\T^n)} \leq \e^\frac12.
	\end{align}
\end{theo}
The assumptions \eqref{t1e1} and \eqref{t1e4} guarantee the existence and uniqueness
of solutions to the linearized problem of \eqref{first} where a given function depends on $\nabla u$.
\eqref{t1e3} is necessary for the existence of solutions to elliptic equations on the torus.
To our knowledge, prescribed mean curvature problems on the torus in the general dimension have been insufficiently studied. However, we have proved the existence of the solution under natural assumptions.

The following is method of proof. We first find the conditions of $H$ for the linearized problem of \eqref{first}
$- \di \left(\frac{\nabla u}{\sqrt{1+|\nabla v|^2}} \right)= H $ to have a unique solution. If we add a suitable constant term for any $v$, the function $\nu(\nabla v) \cdot g(x,v(x)) $
satisfies the conditions.
By estimating the norm of this solution with $g$, the mapping $T(v)=u$ has a fixed point
using a fixed-point theorem, and Theorem \ref{MT} follows.

	\section{Proof of Theorem \ref{MT}}
A theorem that holds in Euclidean space also holds on a torus, as we consider a function on a torus to be a periodic function in Euclidean space.

Let $X(\T^n)$ be a function space on $\T^n$. We define a subspace $X_{ave}(\T^n) \subset X(\T^n)$ as
\[X_{ave} := 
\{ w \in X; \int_{\T^n} w =0\}.\]

\begin{theo} \label{theo1}
	Suppose $v \in C^1(\T^n)$ and $H \in L^2_{ave}(\T^n) $.
	Then, there exists a unique function $u \in W^{1,2}_{ave}(\T^n)$ such that
	\[
	\int_{\T^n} \frac{\nabla u \cdot \nabla \phi}{\sqrt{1+|\nabla v|^2}}
	= \int_{\T^n} H \phi
	\]
	for all $\phi \in W^{1,2}(\T^n)$.
\end{theo}
\begin{proof}
	We define a function $B:W^{1,2}_{ave}(\T^n) \times W^{1,2}_{ave}(\T^n) \to \R $ as
	\[
	B[w_1,w_2,v] := \int_{\T^n} \frac{\nabla w_1 \cdot \nabla w_2}{\sqrt{1+|\nabla v|^2}}.
	\]
	By the  H\"{o}lder inequality, we obtain
	\begin{align}
	|B[w_1,w_2,v] | &\leq \int_{\T^n} |\nabla w_1| |\nabla w_2| \nonumber \\
	&\leq \| \nabla  w_1 \|_{L^2(\T^n)} \| \nabla w_2 \|_{L^2(\T^n)} \nonumber \\
	&\leq \| w_1 \|_{W^{1,2}(\T^n)} \| w_2 \|_{W^{1,2}(\T^n)}. \label{t2e1}
	\end{align}
Using the Poincar\'{e} inequality, we have
\begin{align}
|B[w,w,v]| &\geq \frac{1}{\sqrt{1+\|v\|_{C^1(\T^n)}^2}} 
\|\nabla w\|_{L^2(\T^n)}^2 \nonumber\\
&\geq \frac{1}{\sqrt{1+\|v\|_{C^1(\T^n)}^2}} 
\|\nabla w\|_{W^{1,2}(\T^n)}^2. \label{t2e2}
\end{align}
By \eqref{t2e1}, \eqref{t2e2}, and the Lax--Milgram theorem,
for any $H \in L^2_{ave}(\T^n)$,
there exists a unique function $u \in W^{1,2}_{ave}(\T^n)$ such that
\begin{align}
\int_{\T^n} \frac{\nabla u \cdot \nabla \psi}{\sqrt{1+|\nabla v|^2}}
= \int_{\T^n} H \psi \label{t2e3}
\end{align}
for all $\psi \in W^{1,2}_{ave}(\T^n)$.
For any $\phi \in W^{1,2}(\T^n) $,
we define $c_\phi:= \int_{\T^n} \phi$ and $\tilde{\phi}:=\phi -c_\phi \in W^{1,2}_{ave}(\T^n)$.
By \eqref{t2e3} and $H \in L^2_{ave}(\T^n)$, we obtain
\begin{align}
\int_{\T^n} \frac{\nabla u \cdot \nabla \phi}{\sqrt{1+|\nabla v|^2}}
&=\int_{\T^n} \frac{\nabla u \cdot \nabla \tilde{\phi}}{\sqrt{1+|\nabla v|^2}} \nonumber\\
&= \int_{\T^n} H \tilde{\phi}\nonumber\\
&= \int_{\T^n}H \phi.
\end{align}
Thus, Theorem \ref{theo1} follows.
\end{proof}
We define a mollifier as follows.
\begin{eqnarray}
\eta(x) :=
\begin{cases}
C \exp\left(\frac{1}{|x|^2-1} \right) & \mbox{ for }  |x|<1 \\
0 & \mbox{ for }  |x| \geq 1,
\end{cases}
\nonumber
\end{eqnarray}
where the constant $C>0$ is selected such that $\int_{\R^{n+1}} \eta=1$. We define $\eta_\lambda(x) :=\frac{1}{\lambda^n}\eta(\frac{x}{\lambda})$.
For any $f \in L^2(\T^n \times (-1,1))$ and $x^{n+1} \in (-1+\lambda,1-\lambda)$,
\begin{align*}
f_\lambda (x,x^{n+1}):
&= \int_{\T^n \times (-1,1)}  \eta_\lambda(x-y,x^{n+1}-y^{n+1}), 
f(y,y^{n+1})\ dy\\
&= \int_{B^{n+1}(0,\lambda)} \eta_\lambda(y,y^{n+1}) f(x-y,x^{n+1}-y^{n+1})\ dy,
\end{align*}
where $B^{n+1}(x,\lambda)$ is an open ball with center $x$ and radius $\lambda$ in $\T^n \times \R$.
Moreover, for any $g \in  W^{1,p}(\T^n \times (-1,1);\R^{n+1})$, we define
$g_\lambda := (g^1_\lambda,\ldots,g^n_\lambda,g^{n+1}_\lambda)=(g'_\lambda,g^{n+1}_\lambda).$

\begin{lemm} \label{lemma2}
Fix $\beta_{\be \label{be3}}>0$ and $0<\lambda<1 $. Suppose $v \in C^{1}(\T^n)$ satisfies $\|v\|_{C^{1}(\T^n)}< \beta_{\ref{be3}}$ and $g \in W^{1,p}(\T^n \times (-1,1);\R^{n+1})$
satisfies $\partial_{n+1} g^{n+1}(x,x^{n+1})>\beta_{\ref{be3}}| \partial_{n+1} g'(x,x^{n+1}) |$.
For any positive constant $c_{\cc \label{t3c1}}>0$, if $v(\T^n)+c_{\ref{t3c1}}\subset (-1+\lambda,1-\lambda)$,
	\begin{align}
	\int_{\T^n} \nu (\nabla v) \cdot g_\lambda  (x , v) <
	\int_{\T^n} \nu (\nabla v) \cdot g_\lambda  (x , v+c_{\ref{t3c1}}).
	\end{align}
\end{lemm}
\begin{proof} From the assumptions, we compute
	\begin{align}
	&\int_{\T^n} \nu (\nabla v) \cdot (g_\lambda  (x , v+c_{\ref{t3c1}}) -g_\lambda  (x , v) ) \nonumber \\
	= &\int_{\T^n} \frac{1}{\sqrt{1+|\nabla v|^2}}
	\int^{v+c_{\ref{t3c1}}}_v -\nabla v \cdot \partial_{n+1} g'_\lambda  (x,t)
	+\partial_{n+1} g^{n+1}_\lambda  (x,t) \ dt \nonumber \\
	\geq &\int_{\T^n} \frac{1}{\sqrt{1+|\nabla v|^2}}
	\int^{v+c_{\ref{t3c1}}}_v -\beta_{\ref{be3}} |\partial_{n+1} g'_\lambda (x,t)|
	+\partial_{n+1} g^{n+1}_\lambda  (x,t) \ dt \nonumber \\
	\geq &\int_{\T^n} \frac{1}{\sqrt{1+|\nabla v|^2}}
	\int^{v+c_{\ref{t3c1}}}_v \int_{\T^n \times (-1,1)} \eta_\lambda (x-y,t-y_{n+1}) \times
	\nonumber \\
	&\{-\beta_{\ref{be3}} |\partial_{n+1} g' (y,y_{n+1})|
	+\partial_{n+1} g^{n+1} (y,y_{n+1})\} \ dt \nonumber \\
	>& 0.
	\end{align}
	Lemma \ref{lemma2} follows.
\end{proof}
\begin{lemm} \label{lm3}
		Suppose $g \in W^{1, p}(\T^n \times (-1,1)$
		and $v \in C^1(\T^n)$ with $\|v \|_{C^1(\T^n)} \leq\frac{7}{16}$.
		Let $q=\frac{np}{n+1-p}$. Then, there exists a constant $c_{\cc \label{t4c1}}=c_{\ref{t4c1}}(n,p)>0$ such that,
		 if $\lambda<\frac18$,
		\begin{align}
		\|g_\lambda(\cdot,v(\cdot))\|_{L^{q}(\T^n)} \leq c_{\ref{t4c1}} 
		\|g\|_{W^{1,p}(\T^n \times  (-1,1))}.
		\end{align}
\end{lemm}
\begin{proof}
	By the same proof as in \cite[Lemma 2.3]{T20}, we obtain
	\begin{align}
	\|g_\lambda(\cdot,v(\cdot))\|_{L^{q}(\T^n)} \leq c_{\ref{lc0}} 
	\|g_\lambda\|_{W^{1,p}(\T^n \times  (-\frac78,\frac78))}, \label{lt1}
	\end{align}
	where $c_{\cc \label{lc0}} =c_{\ref{lc0}}(n,p)>0$. 
	Using the  H\"{o}lder inequality, we obtain
\begin{align}
	&\int_{\T^n\times (-\frac78,\frac78)} |g_\lambda|^p  \ dx\nonumber \\ 
	\leq&\int_{\T^n\times (-\frac78,\frac78)} \left(\int_{B^{n+1}(x, \lambda)}
	\eta_\lambda^{1-\frac1p+\frac1p}(x-y,x^{n+1}-y^{n+1})
	|g(y,y^{n+1})| \ dy \right)^p \ dx \nonumber \\
	\leq & \int_{\T^n\times (-\frac78,\frac78)} \left(\int_{B^{n+1}(x, \lambda)}
	\eta_\lambda(x-y,x^{n+1}-y^{n+1})
	|g(y,y^{n+1})|^p \ dy \right) \ dx \nonumber \\
	\leq& \int_{\T^n \times (-1,1)} |g(y,y^{n+1})|^p
	\left(\int_{B^{n+1}(y, \lambda)}\eta_\lambda(x-y,x^{n+1}-y^{n+1}) \ dx \right) \ dy
	\nonumber \\
	=& \int_{\T^n \times (-1,1)} |g(y,y^{n+1})|^p \ dy. \label{lt2}
	\end{align}
	We can show that $\|\nabla g_\lambda\|_{L^p(\T^n \times  (-\frac78,\frac78))} \leq 
	\|\nabla g\|_{L^p(\T^n \times  (-1,1))}$ in the exact same manner, and
	Lemma \ref{lm3} follows by \eqref{lt1} and \eqref{lt2}.
\end{proof}
	\begin{theo} \label{T4}
		Suppose $v \in C^{1}(\T^n)$ and $g \in W^{1,p}(\T^n \times (-1,1);\R^{n+1})$.
		Then, there exist constants $\e_{\ce \label{ce3}}=\e_{\ref{ce3}}(n,p)>0$, if $\lambda <\frac18 $, $\e<\e_{\ref{ce3}}$,
		$ \|v\|_{C^1(\T^n) }\leq\e^\frac12$, and
		$g$ satisfies \eqref{t1e1}--\eqref{t1e3}. Then,
		there exist a unique function $u \in W^{1,2}_{ave}(\T^n)$ and
		a unique constant $-\frac14<c_v <\frac14$ such that
		\begin{align}
		\int_{\T^n} \frac{\nabla u \cdot \nabla \phi}{\sqrt{1+|\nabla v|^2}}
		= \int_{\T^n}  \nu(\nabla v) \cdot g_\lambda(x, v+c_v) \phi \label{t4e0}
		\end{align}
		for all $\phi \in W^{1,2}(\T^n)$.
	\end{theo}
\begin{proof}
	We define
	\begin{align}
	F(t):= \int_{\T^n}  \nu(\nabla v) \cdot g_\lambda(x, v+t). 
	\end{align}
The function $F$ is continuous. Suppose that $\e <\frac{1}{16^2}$. We will consider that
 the domain of $F$ is $[-\frac14,\frac14]$.
By the mean value theorem,
there exists a constant $c_{\cc \label{t4c2}}=c_{\ref{t4c2}}(n,p)>0$ such that
	\begin{align}
F \left( \frac14 \right)&= \int_{\T^n} \left( \nu(\nabla v)-\nu(0)+\nu(0)\right) \cdot g_\lambda\left(x, v+\frac14\right)  \nonumber \\
&\geq -c_{\ref{t4c2}} \|v\|_{C^1(\T^n)} \left\|g_\lambda\left(\cdot,v(\cdot)+\frac14\right)\right\|_{L^{q}(\T^n)}
+\int_{\T^n}g_\lambda^{n+1}\left(x, v+\frac14\right). \label{t4e1}
\end{align}
By Lemma \ref{lm3} and $\|v+\frac14\|_{C^1(\T^n)}\leq \frac{5}{16}$, we obtain
\begin{align}
\left\|g_\lambda\left(\cdot,v(\cdot)+\frac14\right)\right\|_{L^{q}(\T^n)}
\leq c_{\ref{t4c1}} \|g\|_{W^{1,p}\left(\T^n \times  \left(-1,1\right)\right)}. \label{t4e3}
\end{align}
By \eqref{t1e4} and \eqref{t1e3},
there exists a constant $c_{\cc \label{t4c3}}=c_{\ref{t4c3}}(n)>0$  such that
\begin{align}
&\int_{\T^n}g_\lambda^{n+1}\left(x, v+\frac14\right) \nonumber \\
=&\int_{\T^n}\int_{B^{n+1}(0,\lambda)} \eta_\lambda(y,y^{n+1}) g^{n+1}\left(x-y,v+\frac14-y^{n+1}\right) \ dydx \nonumber\\
>&\int_{\T^n}\int_{B^{n+1}(0,\lambda)} \eta_\lambda(y,y^{n+1}) g^{n+1}\left(x-y, \frac{1}{16}\right)\ dy dx \nonumber\\
> &\int_{\T^n}\int_{B^{n+1}(0,\lambda)} \eta_\lambda(y,y^{n+1})
 \left(g^{n+1}\left(x-y, 0\right)+\frac{\e}{16}\right)\ dydx \nonumber\\
 >& \frac{c_{\ref{t4c3}}}{16}\e. \label{t4e2}
\end{align}
By \eqref{t1e1}, \eqref{t4e1}--\eqref{t4e2}, and $ \|v\|_{C^1(\T^n) }< \e^\frac12$,
if $\e<\left(\frac{c_{\ref{t4c3}}}{16c_{\ref{t4c1}}c_{\ref{t4c2}}} \right)^6=:\e_{\ref{ce3}}(n,p)$,
 \begin{align}
 F \left( \frac14 \right)
 &>- c_{\ref{t4c1}} c_{\ref{t4c2}} \|v\|_{C^1(\T^n)} \|g_\lambda\|_{W^{1,p}(\T^n \times  (-1,1))}
 +\frac{c_{\ref{t4c3}}}{16}\e \nonumber\\
 &> -c_{\ref{t4c1}}c_{\ref{t4c2}} \e^\frac76 +\frac{c_{\ref{t4c3}}}{16}\e \nonumber\\
 &> \e\left(-c_{\ref{t4c1}}c_{\ref{t4c2}} \e^\frac16+\frac{c_{\ref{t4c3}}}{16}\right) \nonumber \\
 &>0.
 \end{align}
Similarly, we can show that
$F \left( -\frac14 \right)<0$.
By Lemma \ref{lemma2} and the mean value theorem,
there exists a unique constant $-\frac14<c_v <\frac14$ that satisfies $F(c_v)=0$.
Using Theorem \ref{theo1}, Theorem \ref{T4} follows.
\end{proof}
Let us define an operator $T: \mathcal{A}(s)\to W^{1,2}_{ave}(\T^n)\times \left[-\frac14,\frac14\right]$ by $T(v)=(T_1(v),T_2(v)):=(u,c_v)$ that satisfies
\eqref{t4e0}, where $\mathcal{A}(s)  := \{w \in W^{2,q}_{ave}(\T^n):\|w\|_{W^{2,q}(\T^n)} \leq s\}$.
\begin{theo} \label{T5}
	There exist constants $\e_{\ce \label{ce5}}=\e_{\ref{ce5}}(n,p)>0$ and
	$c_{\cc \label{t5c1}}=c_{\ref{t5c1}}(n,p) >0$,
	if $\lambda < \frac18$, $\e<\min\left\{\e_{\ref{ce3}},\e_{\ref{ce5}} \right\}$, $v \in \mathcal{A}(\e^\frac12)$, 
	and $g \in W^{1,p}(\T^n \times (-1,1);\R^{n+1})$ satisfies \eqref{t1e1}--\eqref{t1e3}.
	Then, 
	\begin{align}
	\|T_1(v)\|_{W^{2,q}(\T^n)} \leq c_{\ref{t5c1}}\|g\|_{W^{1,p}\left(\T^n \times  \left(-1,1\right)\right)}. \label{t5e1}
	\end{align}
\end{theo}
\begin{proof}
	We first assume that $v \in C^\infty(\T^n) \cap \mathcal{A}(\e^\frac12)$.
	Using \cite[Corollary 8.11]{G83}, we obtain $T_1(v) \in C^\infty(\T^n)$; thus, we can
	rewrite \eqref{t4e0} as
	\begin{align}
	\frac{\Delta T_1(v)}{\sqrt{1+|\nabla v|^2}}+ \nabla T_1(v) \cdot\nabla
	 \left(\frac{1}{\sqrt{1+|\nabla v|^2}}\right) = -\nu(\nabla v)\cdot g_\lambda(x,v+T_2(v)).
	\end{align}
	Using \cite[Theorem 9.11]{G83},
	we find that there exists a constant 
	$c_{\cc \label{t5c3}}=c_{\ref{t5c3}}(n,p) >0$ such that
	\begin{align}
	\|T_1(v)\|_{W^{2,q}(\T^n)} \leq &c_{\ref{t5c3}}
	\Biggr( \|T_1(v)\|_{L^{q}(\T^n)}+ \|\nu(\nabla v)\cdot g_\lambda(x,v+T_2(v))\|_{L^{q}(\T^n)}
	 \nonumber \\&+ \left\|\nabla T_1(v) \cdot \nabla
	 \left(\frac{1}{\sqrt{1+|\nabla v|^2}}\right) \right\|_{L^{q}(\T^n)}\Biggr) . \label{t5e5}
	\end{align} 
Using Lemma \ref{lm3}, we obtain
\begin{align}
\|\nu(\nabla v)\cdot g_\lambda(x,v+T_2(v))\|_{L^{q}(\T^n)}
\leq c_{\ref{t4c1}} \|g\|_{W^{1,p}\left(\T^n \times  \left(-1,1\right)\right)}.
\end{align}
Using the Sobolev inequality,
we find that there exists a constant 
$c_{\cc \label{t5c2a}}=c_{\ref{t5c2a}}(n,p) >0$ such that
\begin{align}
	&\left\|\nabla T_1(v) \cdot \nabla
	\left(\frac{1}{\sqrt{1+|\nabla v|^2}}\right) \right\|_{L^{q}(\T^n)} \nonumber \\
	\leq&\|T_1(v)\|_{C^1(\T^n)}
	\left\|\nabla\left(\frac{1}{\sqrt{1+|\nabla v|^2}}\right) \right\|_{L^{q}(\T^n)} \nonumber \\
	\leq &c_{\ref{t5c2a}}\|T_1(v)\|_{W^{2,q}(\T^n)} \|v\|_{W^{2,q}(\T^n)}.
\end{align}
Next, we estimate the term $\|T_1(v)\|_{L^{q}(\T^n)}$. If $q \leq 2$, then, by \eqref{t2e2} and
Lemma \ref{lm3}, we obtain 
\begin{align}
\|T_1(v)\|_{L^{q}(\T^n)} &\leq c_{\cc }(n,p)\|T_1(v)\|_{L^{2}(\T^n)} \nonumber \\
&\leq c_{\cc }(n,p) B[T_1(v),T_1(v),v]^{\frac12}. \nonumber\\
& = c_\arabic{c} \left(
\int_{\T^n} \frac{\nabla T_1(v) \cdot \nabla T_1(v)}{\sqrt{1+|\nabla v|^2}} \right)^\frac12
 \nonumber \\
 & = c_\arabic{c} \left(
\int_{\T^n}  \nu(\nabla v) \cdot g_\lambda(x, v+T_2(v)) T_1(v) \right)^\frac12
 \nonumber\\
 &\leq c_{\cc }(n,p) \|g\|^\frac12_{W^{1,p}(\T^n)} \|T_1(v)\|^\frac12_{L^\infty(\T^n)} \nonumber \\
 &\leq c_{\cc}(n,p)\|g\|_{W^{1,p}(\T^n)}+\frac{1}{4c_{\ref{t5c3}}}  \|T_1(v)\|_{W^{2,q}(\T^n)}.
\label{t5e3}
\end{align}
If $q>2$, by \eqref{t5e3} and the Riesz--Thorin theorem, we obtain
\begin{align}
\|T_1(v)\|_{L^{q}(\T^n)} &\leq \|T_1(v)\|_{L^{2}(\T^n)}^\frac1q 
\|T_1(v)\|_{L^{2}(\T^n)}^{1-\frac1q }\nonumber \\
&\leq c_{\cc}(n,p) \|g\|_{W^{1,p}(\T^n)}^\frac{1}{2q}
\|T_1(v)\|_{L^\infty(\T^n)}^{\frac{1}{2q}+1-\frac1q}
 \nonumber\\
&\leq  c_{\cc}(n,p)\|g\|_{W^{1,p}(\T^n)}+\frac{1}{4c_{\ref{t5c3}}}  \|T_1(v)\|_{W^{2,q}(\T^n)}. \label{t5e4}
\end{align}
By \eqref{t5e5}--\eqref{t5e4}, there exists a constant $c_{\cc \label{t5c5}}=c_{\ref{t5c5}}(n,p)>0$
	such that
	\begin{align}
\|T_1(v)\|_{W^{2,q}(\T^n)} 
\leq&
c_{\ref{t5c5}}(\|g\|_{W^{1,p}\left(\T^n \times  \left(-1,1\right)\right)}+
\|T_1(v)\|_{W^{2,q}(\T^n)} \|v\|_{W^{2,q}(\T^n)})
\nonumber \\&+\frac{1}{4}  \|T_1(v)\|_{W^{2,q}(\T^n)}.
	\end{align}
	If $\e < \frac{1}{16c^2_{\ref{t5c5}}}$, we obtain
	\begin{align}
	\|T_1(v)\|_{W^{2,q}(\T^n)} \leq 2c_{\ref{t5c5}}\|g\|_{W^{1,p}\left(\T^n \times  \left(-1,1\right)\right)}. \label{t5e6}
	\end{align}
For the general case of $v \in W^{2,q}(\T^n)$,
suppose that $\{v_m \}_{m \in \Na} \in C^\infty(\T^n)$ converges to $v$ in the sense of $C^1(\T^n)$.
By \eqref{t5e6}, there exists a subsequence $\{v_{m_k}\}_{k \in \Na} \subset    \{v_m \}_{m \in \Na}$ 
such that $T_1(v_{m_k})$ converges to a function $w_\infty \in W^{2,q}(\T^n)$
in the sense of  $C^1(\T^n)$ and
$T_2(v_{m_k})$ converges to a constant $d_\infty \in \left[-\frac14,\frac14\right]$.
For any $\phi \in W^{1,2}(\T^n)$, we obtain
\begin{align}
&\int_{\T^n}  \nu(\nabla v) \cdot g_\lambda(x, v+d_\infty)  \phi -
\nu(\nabla v_{m_k}) \cdot g_\lambda(x, v_{m_k}+T_2(v_{m_k}))  \phi \nonumber \\
\leq&\int_{\T^n}|\phi||\nu(\nabla v) - \nu(\nabla v_{m_k})| 
|g_\lambda(x, v_{m_k}+T_2(v_{m_k}))|\nonumber \\&+ \int_{\T^n} |\phi|  \left|\int_{v_{m_k}+T_2(v_{m_k})}^{v+d_\infty}
\partial_{n+1} g_\lambda(x, s) \right|  \nonumber \\
\to& 0 \quad (k \to \infty)  \label{t5e7}
\end{align}
and
\begin{align}
&\int_{\T^n} \frac{\nabla w_\infty \cdot \nabla \phi}{\sqrt{1+|\nabla v|^2}}
-\frac{\nabla T_1(v_{m_k}) \cdot \nabla \phi}{\sqrt{1+|\nabla v_{m_k}|^2}} \nonumber\\
\leq&\int_{\T^n} \frac{(\nabla w_\infty-\nabla T_1(v_{m_k})) \cdot \nabla \phi}{\sqrt{1+|\nabla v|^2}}
\nonumber \\
&+ \int_{\T^n} (\nabla T_1(v_{m_k}) \cdot \nabla \phi)\left(
 \frac{1}{\sqrt{1+|\nabla v|^2}} -\frac{1}{\sqrt{1+|\nabla v_{m_k}|^2}}\right) \nonumber\\
\to& 0 \quad (k \to \infty) . \label{t5e8}
\end{align}
By \eqref{t5e7} and \eqref{t5e8}, we obtain
\begin{align}
&\int_{\T^n} \frac{\nabla w_\infty \cdot \nabla \phi}{\sqrt{1+|\nabla v|^2}}
-\nu(\nabla v) \cdot g_\lambda(x, v+d_\infty)  \phi  \nonumber \\
=&\lim_{k\to \infty} \int_{\T^n} \frac{\nabla T_1(v_{m_k}) \cdot \nabla \phi}{\sqrt{1+|\nabla v_{m_k}|^2}}
-\nu(\nabla v_{m_k}) \cdot g_\lambda(x, v_{m_k}+T_2(v_{m_k}))\phi \nonumber\\
=0, \label{t5e10}
\end{align}
that is, $T(v)=(w_\infty, d_\infty)$. By \eqref{t5e6} and \eqref{t5e10}, Theorem \ref{T5} follows.

\end{proof}
Next, we write the fixed-point theorem, which is needed later (\cite[Theorem 1]{A84}).
An operator $T:X\to A$ is considered weakly sequentially continuous if,
for every sequence $\{x_m\}_{m\in \Na} \subset X$ and $x_\infty \in X$ such that
$x_m$ weakly converges to $x_\infty$, $T(x_m)$ weakly converges to $T(x_\infty)$.
\begin{theo} \label{th4}
	Let $X$ be a metrizable, locally convex topological vector space and $\Omega$ be a weakly compact convex subset of $X$. Then, any weakly
	sequentially continuous map $T: \Omega \to \Omega$ has a fixed point.
\end{theo}
We first prove Theorem \ref{MT} in the case of $ g_\lambda $.
\begin{theo} \label{t7}
There exists a constant $\e_{\ce \label{ce6}}=\e_{\ref{ce6}}(n,p)>0$,
if $\lambda < \frac18$, $\e<\e_{\ref{ce6}}$,
and $g \in W^{1,p}(\T^n \times (-1,1);\R^{n+1})$ satisfies \eqref{t1e1}--\eqref{t1e3}.
Then, there exists a function $u_\lambda \in W^{2,q}(\T^n)$ such that
\begin{align}
- \di \left(\frac{\nabla u_\lambda}{\sqrt{1+|\nabla u_\lambda|^2}} \right)= \nu(\nabla u_\lambda) \cdot g_\lambda(x,u_\lambda(x)) \quad 
\mathrm{on} \  \T^n. \label{t6e1}
\end{align}
\end{theo}
\begin{proof}
	The set $W^{2,q}(\T^n)$ is a metrizable, locally convex topological vector space, and
	the set $\mathcal{A}(\e^\frac12)$ is a weakly compact convex subset of $W^{2,q}(\T^n)$.
	By \eqref{t1e1} and Theorem \ref{T5}, if 
	$\e<\min\{ \e_{\ref{ce3}},{\e_{\ref{ce5}}},  c_{\ref{t5c1}}^{-6}\}=:\e_{\ref{ce6}} $,
	we have
	\begin{align}
	\|T_1(v)\|_{W^{2,q}(\T^n)} &\leq c_{\ref{t5c1}}\|g\|_{W^{1,p}\left(\T^n \times  \left(-1,1\right)\right)} \nonumber\\
	&\leq  c_{\ref{t5c1}} \e^\frac16 \e^\frac12 \nonumber \\
	&\leq \e^\frac12 \quad 
	\mathrm{for \  any} \   v \in \mathcal{A}(\e^\frac12) , \label{t7e1}
	\end{align}
	that is,
	 $T_1(\mathcal{A}(\e^\frac12)) \subset \mathcal{A}(\e^\frac12)$. 
	Suppose that $\{v_m \}_{m \in \Na}$ weakly converges to $v_\infty$ in the sense of $W^{2,q}(\T^n)$.
	According to Theorem \ref{T5}, there exists a subsequence $\{v_{m_k}\}_{k \in \Na} \subset    \{v_m \}_{m \in \Na}$ 
	such that $T_1(v_{m_k})$ weakly converges to a function $w_\infty \in W^{2,q}(\T^n)$
	in the sense of  $W^{2,q}(\T^n)$ and
	$T_2(v_{m_k})$ converges to a constant $d_\infty \in \left[-\frac14,\frac14\right]$.
	By the same argument \eqref{t5e7}--\eqref{t5e10}, for any $\phi \in W^{2,q}(\T^n)$,
	\begin{align}
	\int_{\T^n} \frac{\nabla w_\infty \cdot \nabla \phi}{\sqrt{1+|\nabla v_\infty|^2}}
	-\nu(\nabla v_\infty) \cdot g_\lambda(x, v_\infty+d_\infty)  \phi  =0,
	\end{align}
	that is, we obtain $ \lim_{k\to \infty}T_1(v_{m_k}) = T_1(v_\infty) $
	by the uniqueness of solution of Theorem \ref{T4}.
	Therefore, every convergent subsequence of $\{T_1(v_m) \}$
	converges to $T_1(v_\infty)$, and $T_1$ is a weakly sequentially continuous map.
	Using Theorem \ref{th4}, we obtain a function
	$v_\lambda \in W^{2,q}_{ave}(\T^n)$ satisfying 
	\begin{align}
		- \di \left(\frac{\nabla v_\lambda}{\sqrt{1+|\nabla v_\lambda|^2}} \right)= \nu(\nabla v_\lambda) \cdot g_\lambda(x,v_\lambda(x)+T_2(v_\lambda)) \quad 
		\mathrm{on} \  \T^n,
	\end{align}
that is, $u_\lambda:= v_\lambda +T_2(v_\lambda)\in W^{2,q}(\T^n)$ satisfying  \eqref{t6e1}.
\end{proof}
\begin{proof}[Proof of Theorem \ref{MT}]
	Suppose $u_{\lambda} \in W^{2,q}(\T^n)$ satisfies $\eqref{t6e1}$.
	By Theorem \ref{T5}, there exists a convergent subsequence
	 $\{u_{\lambda_k}\}_{k \in \Na} \subset \{u_{\lambda}\}_{0<\lambda<\frac18}$
	  with a limit $u_\infty \in  W^{2,q}(\T^n)$ in the sense of $C^1(\T^n)$ and $\lambda_k \to 0$. We show that $u_\infty$ satisfies \eqref{t0e1}. For any $\phi \in W^{1,2}(\T^n)$, we obtain
	  \begin{align}
	  &\int_{\T^n} -\di \left(\frac{\nabla u_{\lambda_k}}{\sqrt{1+|\nabla u_{\lambda_k}|^2}} 
	   -\frac{\nabla u_\infty}{\sqrt{1+|\nabla u_\infty|^2}}\right)\phi \nonumber\\
	   =&\int_{\T^n} \left(\frac{\nabla u_{\lambda_k}}{\sqrt{1+|\nabla u_{\lambda_k}|^2}} 
	   -\frac{\nabla u_\infty}{\sqrt{1+|\nabla u_\infty|^2}}\right) \cdot \nabla \phi \nonumber\\
	   \to&0. \label{t6e2}
	  \end{align}
  Using Lemma \ref{lm3}, we have
	  \begin{align}
	  &\int_{\T^n} \nu(\nabla u_{\lambda_k}) \cdot g_{\lambda_k}(x,u_{\lambda_k})
	  -\nu(\nabla u_{\infty}) \cdot g(x, u_{\infty}) \nonumber \\
	  =&\int_{\T^n} (\nu(\nabla u_{\lambda_k}) -\nu(\nabla u_{\infty}))\cdot g_{\lambda_k}(x,u_{\lambda_k}) \nonumber \\
	  &+\int_{\T^n} \nu(\nabla u_{\infty}) \cdot(g_{\lambda_k}(x, u_{\lambda_k})-
	  g(x,u_{\lambda_k})) \nonumber \\
	  &+\int_{\T^n}\nu(\nabla u_{\infty}) \cdot(g(x,u_{\lambda_k})-
	  g(x, u_{\infty})) \nonumber \\
	  =&c_{\ref{t4c1}}\|\nu(\nabla u_{\lambda_k}) -\nu(\nabla u_{\infty})\|_{C^0(\T^n)}\|g_{\lambda_k}\|_{W^{1,p}(\T^n \times  \left(-1,1\right))} \nonumber \\
	  &+c_{\ref{t4c1}}\|g_{\lambda_k}-g_\infty\|_{W^{1,p}(\T^n \times  \left(-1,1\right))}
	   \nonumber \\
	   &+ \int_{\T^n} \left|\int^{u_{\lambda_k}}_{u_{\infty}} \partial_{n+1} g(x,s)\right|\nonumber\\
	   \to&0.\label{t6e3}
  \end{align}
  By \eqref{t6e2} and \eqref{t6e3}, we obtain
  \begin{align*}
  &\int_{\T^n} -\di \left(\frac{\nabla u_\infty}{\sqrt{1+|\nabla u_\infty|^2}}\right)\phi-\nu(\nabla u_{\infty}) \cdot g(x, u_{\infty})\phi \\
  =&\lim_{k\to \infty} \int_{\T^n}
 -\di \left(\frac{\nabla u_{\lambda_k}}{\sqrt{1+|\nabla u_{\lambda_k}|^2}}\right)\phi-\nu(\nabla u_{\lambda_k}) \cdot g_{\lambda_k}(x, u_{\lambda_k})\phi \\
 \to&0.
  \end{align*}
  Thus, $u_\infty$ satisfies \eqref{t0e1}
  using the fundamental lemma of the calculus of variations.
  By \eqref{t7e1}, we obtain
  \begin{align}
  \left\|u_\infty - \int_{\T^n}u_\infty(y) \ dy \right\|_{W^{2,q}(\T^n)} \leq \e^\frac12,
  \end{align}
  and Theorem \ref{MT} follows.
  
\end{proof}

\end{document}